\documentclass[12pt]{article}
\usepackage{graphicx}
\usepackage{amsmath}
\usepackage{amssymb}
\usepackage{theorem}

\numberwithin{equation}{section}

\newtheorem{thm}{Theorem}[section]
\newtheorem{lemma}[thm]{Lemma}
\newtheorem{prop}[thm]{Proposition}
\newtheorem{cor}[thm]{Corollary}
{\theorembodyfont{\rmfamily}

\newtheorem{rmk}[thm]{Remark}
}

\newcommand{\qed}{\hfill \mbox{\raggedright \rule{.07in}{.1in}}}
 
\newenvironment{proof}{\vspace{1ex}\noindent{\bf
Proof}\hspace{0.5em}}{\hfill\qed\vspace{1ex}}
\newenvironment{pfof}[1]{\vspace{1ex}\noindent{\bf Proof of
#1}\hspace{0.5em}}{\hfill\qed\vspace{1ex}}

\newcommand{\R}{{\mathbb R}}
 \newcommand{\Q}{{\mathbb Q}}
\newcommand{\C}{{\mathbb C}}
 \newcommand{\Z}{{\mathbb Z}}
 \newcommand{\E}{{\mathbb E}}

\newcommand{\eps}{{\epsilon}}
 \newcommand{\diam}{\operatorname{diam}}
 
\renewcommand{\Re}{\operatorname{Re}}
\renewcommand{\Im}{\operatorname{Im}}

\title{Broadband Nature of Power Spectra for Intermittent Maps with Summable and Nonsummable Decay of Correlations}

\author{
Georg A. Gottwald\thanks{School of Mathematics and Statistics, University of Sydney, Sydney 2006 NSW, Australia}
\and
Ian Melbourne\thanks{Mathematics Institute, University of Warwick, Coventry, CV4 7AL, UK}
}

\date{31 October 2015}

\begin{document}

\maketitle

 \begin{abstract}
 We present results on the broadband nature of the power spectrum $S(\omega)$, $\omega\in(0,2\pi)$, for a large class of nonuniformly expanding maps with summable and nonsummable decay of correlations.  In particular, we consider a class of intermittent maps $f:[0,1]\to[0,1]$ with $f(x)\approx x^{1+\gamma}$ for $x\approx 0$, where $\gamma\in(0,1)$.
 Such maps have summable decay of correlations when $\gamma\in(0,\frac12)$, and $S(\omega)$ extends to a continuous function on $[0,2\pi]$ by the classical Wiener-Khintchine Theorem.  We show that $S(\omega)$ is typically bounded away from zero for H\"older observables.

 Moreover, in the nonsummable case $\gamma\in[\frac12,1)$, we show that $S(\omega)$ is defined almost everywhere with a continuous extension $\tilde S(\omega)$ defined on $(0,2\pi)$, and $\tilde S(\omega)$ is typically nonvanishing.
 \end{abstract}

\section{Introduction}

Let $f:X\to X$ be a measure preserving transformation of a probability space
$(X,\mu)$ and let $v:X\to\R$ be an $L^2$ observable.
The power spectrum $S:[0,2\pi]\to\R$ is given by
\[
S(\omega)=\lim_{n\to\infty}\frac1n\int_X\Bigl|\sum_{j=0}^{n-1} e^{ij\omega}v\circ f^j\Bigr|^2\,d\mu.
\]
By the Wiener-Khintchine Theorem~\cite{Kampen}, 
$S(\omega)=\sum_{k=-\infty}^\infty e^{ik\omega}\rho(k)$ where
$\rho(k)=\int_X v\circ f^k\,v\,d\mu-\Bigl(\int_X v\,d\mu)^2$ is the autocorrelation function of $v$.
In particular, the power spectrum is analytic if and only if the autocorrelations decay exponentially.
More generally, the power spectrum is well-defined and continuous provided the autocorrelations are summable.

The power spectrum is often used by experimentalists to distinguish periodic and quasiperiodic dynamics (discrete power spectrum with peaks at the harmonics and subharmonics) and chaotic dynamics (broadband power spectra).  
See for example~\cite{GollubSwinney75}. 
In the atmospheric and oceanic sciences and in climate science, power spectra have been widely used to detect variability in particular frequency bands (see, for example, \cite{GhilEtAl02,Dijkstra}). 
Spectral analysis was successful in detecting dominant time scales in teleconnection patterns, revealing intraseasonal variability in time series of the global atmospheric angular momentum \cite{DickeyEtAl91}, interannual variability in the El Ni\~no/Southern Oscillation system \cite{Chen82,JiangEtAl95}, and the Atlantic Multidecadal Variability \cite{DelworthMann00}. 
On millenial temporal scales the power spectrum was instrumental in unraveling dominant cycles in paleoclimatic records \cite{Wunsch00,Wunsch03}. 

Despite this widespread applicability across these disparate temporal scales, there are surprisingly few rigorous results on the nature of power spectra of complex systems. 
The (quasi)periodic case with its peaks at discrete frequencies is well understood. 
The nature of power spectra for chaotic systems was first treated in \cite{Ruelle86} in the case of uniformly hyperbolic (Axiom~A) systems. 
In our previous paper~\cite{MG08}, we considered in more detail the broadband nature of power spectra for chaotic dynamical systems and showed that for certain classes of dynamical systems $f$ and observables~$v$, the power spectrum is bounded away from zero.

The main results in~\cite{MG08} are for nonuniformly expanding/hyperbolic dynamical systems with exponential decay of correlations.  These results are summarised below in Subsection~\ref{sec-MG}.  The current paper is concerned with systems possessing subexponential --- even nonsummable --- decay of correlations.
A prototypical example is the class of Pomeau-Manneville intermittent maps~\cite{PomeauManneville80}, specifically the class
considered in~\cite{LiveraniSaussolVaienti99}.  These are maps $f:[0,1]\to[0,1]$ given by
\begin{align} \label{eq-LSV}
f(x)=\begin{cases} x(1+2^\gamma x^\gamma), & x\in[0,\frac12) \\
2x-1, & x\in[\frac12,1]
\end{cases},
\end{align}
where $\gamma>0$ is a parameter. For $\gamma\in(0,1)$, there is a unique absolutely continuous invariant probability measure $\mu$, and autocorrelations decay at the rate $O(1/n^\beta)$ for H\"older observables 
where $\beta=\gamma^{-1}-1$.
In particular, if $\gamma\in(0,\frac12)$, then the 
autocorrelation function is summable and the Wiener-Khintchine Theorem assures that the power spectrum is well-defined and continuous.
We show that the power spectrum is typically bounded away from zero.
Moreover, we show that the same result holds for all $\gamma\in(0,1)$ provided the H\"older
exponent\footnote{Recall that if $(X,d)$ is a metric space and $\eta\in(0,1]$, then $v:X\to\R$ is $C^\eta$
(H\"older with exponent $\eta$) if $|v|_\eta=\sup_{x\neq y}|v(x)-v(y)|/d(x,y)^\eta<\infty$.}
of $v$ is sufficiently large.  More precisely, we prove:

\begin{thm} \label{thm-LSV}  Suppose that $f:[0,1]\to[0,1]$ is of the type~\eqref{eq-LSV} where $\gamma\in(0,1)$.  
Suppose that $v:[0,1]\to\R$ is $C^\eta$ where $\eta>0$.

If $\gamma\in(0,\frac12)$, 
then the power spectrum is continuous on $(0,2\pi)$ with a continuous extension to $[0,2\pi]$, and is
typically\footnote{Throughout this paper ``typically'' means lying outside a closed subspace of infinite codimension within the Banach space of H\"older observables of a given exponent.  Thus Theorem~\ref{thm-LSV} fails only for infinitely degenerate observables.}
bounded away from zero.

If $\gamma\in[\frac12,1)$, and $\eta>(3\gamma-1)/2$, then the power spectrum is defined almost everywhere with a continuous extension to $(0,2\pi)$.  Typically this extension is nonvanishing.
\end{thm}

\begin{rmk} \label{rmk-LSV}   When $\gamma\in(0,\frac12)$, 
the result is a special case of a more general result, Theorem~\ref{thm-NUE}, stated below. 
The only part of the result that is new for $\gamma\in(0,\frac12)$ is that the power spectrum is typically bounded below.

The case $\gamma\in[\frac12,1)$ depends more strongly on the details of the system, and seems to be a new result in its entirety.
\end{rmk}

\subsection{The results in~\cite{MG08} for systems with exponential decay}
\label{sec-MG}

The simplest case considered in~\cite{MG08} is when $f$ is either a (noninvertible) uniformly expanding map
or a uniformly hyperbolic (Axiom~A) diffeomorphism and
$\Lambda$ is a locally maximal transitive subset of $X$.
Suppose first that $\Lambda$ is mixing.
Then $f:\Lambda\to\Lambda$ has exponential decay of correlations for H\"older observables; in particular, the autocorrelation function of $v$ decays exponentially provided $v$ is H\"older.
In this situation, we showed~\cite[Theorem~1.3]{MG08} that the power spectrum 
is bounded away from zero for typical H\"older observables.

Still in the uniformly expanding/hyperbolic setting, it was shown in~\cite{MG08} that the mixing condition is unnecessary.
There is an integer $q\ge1$ such that (i) $\Lambda$ is a disjoint union
$\Lambda=\Lambda_1\cup\cdots\cup \Lambda_q$, (ii) $f(\Lambda_i)= \Lambda_{i+1}$ (computing subscripts $\bmod\, q$), and (iii) $f^q:\Lambda_i\to\Lambda_i$ is mixing for each $i$.
Moreover, $f^q:\Lambda_i\to\Lambda_i$ has exponential decay of correlations for H\"older observables.  By~\cite{MG08}, the power spectrum is analytic with removable singularities at $2\pi j/q$, $j=0,\dots,q$, and typically bounded away from zero.

Large classes of nonuniformly expanding/hyperbolic systems can be treated in the same manner, namely those modelled by the tower construction of Young~\cite{Young98}.
These systems include the logistic family, H\'enon-like attractors, and planar periodic dispersing billiards.  In such cases, the power spectrum is again analytic (up to removable singularities) and typically bounded away from zero.

\subsection{Systems with subexponential but summable decay}

The current paper is concerned with the case when exponential decay of correlations fails.  Young~\cite{Young99} considers nonuniformly expanding maps with subexponential decay of correlations.   
The precise definition of nonuniformly expanding map is given in Section~\ref{sec-NUE}.
Although the power spectrum is no longer analytic, we may still discuss its boundedness properties.

If the decay is summable, then the power spectrum is continuous and extends to a continuous (and hence bounded) function on $[0,2\pi]$.
We prove:

\begin{thm}  \label{thm-NUE}
Let $f$ be a nonuniformly expanding map with
polynomial decay of correlations at a rate $O(1/n^\beta)$ where $\beta>1$.
Then the power spectrum is bounded away from zero for typical H\"older observables.
\end{thm}

\begin{rmk}
This result was claimed in~\cite[Section~4]{MG08} but the proof sketched there is incomplete (the iterates of $L^k\hat v$ are summable as claimed, but only in $L^p$ spaces, so the step that involves evaluation at periodic data is problematic).  
A full proof is given in this paper.  
Moreover, we consider also the case $\beta\in(0,1]$.
\end{rmk}

The remainder of this paper is organised as follows.
In Section~\ref{sec-NUE}, we give the definition for nonuniformly expanding map, and state Theorem~\ref{thm-main} which implies Theorem~\ref{thm-NUE}.
Also, we show how Theorem~\ref{thm-LSV} follows from Theorem~\ref{thm-main}.
In Section~\ref{sec-proof}, we prove Theorem~\ref{thm-main}.

\section{Nonuniformly expanding maps}
\label{sec-NUE}

Let $(X,d)$ be a locally compact separable bounded metric space with
Borel probability measure $m_0$ and let $f:X\to X$ be a nonsingular
transformation for which $m_0$ is ergodic.
Let $Y\subset X$ be a measurable subset with $m_0(Y)>0$, and
let $\alpha$ be an at most countable measurable partition
of $Y$.    We suppose that there is an $L^1$
{\em return time} function $r:Y\to\Z^+$, constant on each $a\in\alpha$ with
value $r(a)\ge1$, and constants $\lambda>1$, $\eta\in(0,1)$, $C\ge1$,
such that for each $a\in\alpha$,
\begin{itemize}
\item[(1)] $F=f^{r(a)}:a\to Y$ is a measure-theoretic bijection.
\item[(2)] $d(Fx,Fy)\ge \lambda d(x,y)$ for all $x,y\in a$.
\item[(3)] $d(f^\ell x,f^\ell y)\le Cd(Fx,Fy)$ for all $x,y\in a$,
$0\le \ell <r(a)$.
\item[(4)] $g_a=\frac{d(m_0|{a}\circ F^{-1})}{dm_0|_Y}$
satisfies $|\log g_a(x)-\log g_a(y)|\le Cd(x,y)^\eta$ for all
\mbox{$x,y\in Y$}.
\end{itemize}
Such a dynamical system $f:X\to X$ is called {\em nonuniformly expanding}.
The induced map $F=f^r:Y\to Y$ is uniformly expanding and 
there is a unique $F$-invariant probability measure $\mu_Y$ on $Y$ equivalent
to $m_0|_Y$ with density bounded above and below.  Moreover $\mu_Y$ is mixing.
This leads to a unique $f$-invariant probability measure $\mu$ on $X$ equivalent
to $m_0$ (see for example~\cite[Theorem~1]{Young99}).  

We assume throughout that
$\gcd\{r(a)-r(b):a,b\in\alpha\}=1$.  In particular, $\mu$ is mixing.
The assumption is trivially satisfied for the maps~\eqref{eq-LSV} since $\{r(a):a\in\alpha\}=\Z^+$.  Such a restriction is not completely avoidable, since
the power spectrum has a finite number of (removable) singularities~\cite{MG08} when $F$ is not mixing.

The results in~\cite{MG08} apply directly when 
$\mu_Y(y\in Y:r(y)>n)$ decays exponentially.
In this paper, we show that the same results hold when $r\in L^{2+}(Y)$, thus proving 
Theorem~\ref{thm-NUE}.\footnote{Throughout, we write $\phi\in L^{p+}(Y)$ as shorthand for $\phi\in L^{p+\eps}(Y)$ for some $\eps>0$.}
In addition, we obtain results in the case $r\in L^{1+}(Y)$.

Given $v\in L^\infty(X)$ and
$\omega\in[0,2\pi]$, we define 
the {\em induced observable} $V_\omega:Y\to\C$,
\begin{align*}
 V_\omega(y)=\sum_{\ell=0}^{r(y)-1}e^{i\ell\omega}v(f^\ell y).
\end{align*}

\begin{prop}  \label{prop-V}
For all $\omega_0,\omega\in[0,2\pi]$, $a\in\alpha$,
\begin{align*}
 |1_aV_\omega|_\infty\le |v|_\infty r(a), \quad
|1_aV_\omega-1_aV_{\omega_0}|_\infty\le 
2|v|_\infty r(a)^2|\omega-\omega_0|.
\end{align*}
\end{prop}

\begin{proof}
The first estimate is immediate.  Also, for $y\in a$,
\[
|V_\omega(y)-V_{\omega_0}(y)|\le 
\sum_{\ell=0}^{r(a)-1}|e^{i\ell\omega}-e^{i\ell\omega_0}||v|_\infty
\le 2\sum_{\ell=0}^{r(a)-1}|\ell(\omega-\omega_0)||v|_\infty
\le 2|v|_\infty r(a)^2|\omega-\omega_0|,
\]
as required.
\end{proof}

\begin{cor} \label{cor-V}
Let $p\ge1$.  If $r\in L^p(Y)$ then $V_\omega\in L^p(Y)$ for all $\omega\in[0,2\pi]$.
Moreover,
$\omega\mapsto V_\omega$ is a continuous 
map from $[0,2\pi]$ to $L^p(Y)$.
\end{cor}

\begin{proof}
We have 
$|V_\omega|_p^p\le \sum_{a\in\alpha} \mu_Y(a)|1_aV_\omega|^p_\infty
\le \sum_{a\in\alpha} \mu_Y(a) |v|^p_\infty r(a)^p=|v|_\infty^p|r|_p^p<\infty$,
so $V_\omega\in L^p(Y)$.

Next we prove continuity.  Let $\omega_0\in[0,2\pi]$.  Then
\begin{align*}
|V_\omega-V_{\omega_0}|_p^p 
& \le \sum_{a\in\alpha}\mu_Y(a)|1_aV_\omega-1_aV_{\omega_0}|_\infty^p \\
& \le \sum_{a\,:\,r(a)\le R}\mu_Y(a)2^p|v|_\infty^p R^{2p}|\omega-\omega_0|^p
+ \sum_{a\,:\,r(a)>R}\mu_Y(a)2^p|v|_\infty^pr(a)^p  \\
& \le 2^p|v|_\infty^p\Bigl(R^{2p}|\omega-\omega_0|^p + \int_{\{r(a)>R\}}r^p\,d\mu_Y\Bigr).
\end{align*}
Fix $R$ large so that the second term is as small as desired.   For this fixed $R$, the first term converges to zero as $\omega\to\omega_0$, proving continuity at $\omega_0$.
\end{proof}

Define $V^*_\omega:Y\to\C$,
\begin{align*}
 V^*_\omega(y)=\max_{0\le j\le r(y)-1}\Bigl|\sum_{\ell=0}^je^{i\ell\omega}v(f^\ell y)\Bigr|.
\end{align*}
Again $|1_aV^*_\omega|_\infty\le |v|_\infty r(a)$, so
if $r\in L^p(Y)$, then $V^*_\omega\in L^p(Y)$ for all $\omega\in[0,2\pi]$.

We now state our main result; this is proved in Section~\ref{sec-proof}.
\begin{thm} \label{thm-main}
	Let $f:X\to X$ be a nonuniformly expanding map and let
	$v:X\to\R$ be a H\"older observable.  

	(a)  Suppose that $r\in L^{2+}(Y)$.
Then 
the limit $S(\omega)=\lim_{n\to\infty}n^{-1}\int_X|e^{ij\omega}v\circ f^j|^2\,d\mu$ exists for all $\omega\in (0,2\pi)$ and extends to a continuous
function on $[0,2\pi]$.
Typically, $\inf_{\omega\in (0,2\pi)} S(\omega)>0$.

(b) Suppose that $r\in L^a(Y)$ and that
$V^*_\omega\in L^{bp}(Y)$ for all $\omega\in (0,2\pi)$, where
$a\in(1,\infty]$, $1/a+1/b=1$ and $p>2$.
Suppose further that $\omega\mapsto V_\omega$ is a continuous map from $(0,2\pi)$ to $L^2(Y)$.
Then the limit $S(\omega)=\lim_{n\to\infty}n^{-1}\int_X|e^{ij\omega}v\circ f^j|^2\,d\mu$ exists for almost every $\omega\in (0,2\pi)$ and extends to a continuous
function $\tilde S(\omega)$ on $(0,2\pi)$.
Typically, $\tilde S(\omega)$ is nonvanishing on $(0,2\pi)$.
\end{thm}

\begin{rmk}
	(i) Young~\cite{Young99} considers the case where $\mu_Y(r>n)=O(1/n^{\beta+1})$ for some $\beta>0$ and deduces decay of correlations at rate $O(1/n^\beta)$.  The case $\beta>1$ is the setting of Theorem~\ref{thm-NUE}.
It is easily seen that $r\in L^{p+}(Y)$ if and only if $\mu(r>n)=O(1/n^{p+})$, so Theorem~\ref{thm-NUE} is a restatement of Theorem~\ref{thm-main}(a).

\vspace{1ex} \noindent (ii) If $r\in L^2(Y)$, then summable decay of correlations follows from~\cite[Corollary~1.3]{MT14} and hence the Wiener-Khintchine Theorem guarantees that the power spectrum extends to a continuous function on $[0,2\pi]$.  However, our proof that the spectrum is typically bounded below requires that $r\in L^{2+}(Y)$.

\vspace{1ex} \noindent (iii) If $r\not\in L^2(Y)$, then we expect (but have been unable to prove) that $S(\omega)\to\infty$ as $\omega\to0$ and $\omega\to 2\pi$.  It would then follow that typically the power spectrum is bounded below also in Theorem~\ref{thm-main}(b).

\vspace{1ex} \noindent (iv) The proof of Theorem~\ref{thm-main}(b) shows that
the limit $S(\omega)$ exists and is continuous on the set of irrational angles $\omega$.   A different argument, which we have not included, shows that the limit exists also for rational angles $\omega\in(0,2\pi)$ and we conjecture that the resulting function $\omega\to S(\omega)$ is continuous on $(0,2\pi)$.
\end{rmk}

\subsection{Application to intermittent maps}

For the intermittent maps~\eqref{eq-LSV}, it is convenient to take $Y=[\frac12,1]$ and to let $r:Y\to\Z^+$ be the first return time
$r(y)=\inf\{n\ge1:f^ny\in Y\}$.  It is standard that $f$ is a nonuniformly expanding map and that
$r\in L^p(Y)$ for any $p<\frac{1}{\gamma}$.  Hence if $\gamma\in(0,\frac12)$, then Theorem~\ref{thm-main} applies directly.
This completes the proof of Theorem~\ref{thm-LSV} for $\gamma\in(0,\frac12)$.

For $\gamma\in[\frac12,1)$, we still have that $r\in L^{1+}(Y)$.
To apply Theorem~\ref{thm-main}, we require the next result.

\begin{prop}  \label{prop-LSV}
Suppose that $\gamma\in(0,1)$ and $v\in C^\eta$ where $0<\eta<\gamma$.
Let $p< 1/(\gamma-\eta)$.
Then $V^*_\omega\in L^p(Y)$ for all $\omega\in(0,2\pi)$ and $\omega\mapsto V_\omega$ is a continuous map from $(0,2\pi)$ to $L^p(Y)$.
\end{prop}

\begin{proof}
This is analogous to the situation in~\cite[Theorem~4.1]{GMsub} (see also~\cite{Gouezel04}).

Writing $v=(v-v(0))+v(0)$, we may consider the cases $v(0)=0$ and $v\equiv v(0)$ separately.
For $v(0)=0$, we show that $V^*_\omega\in L^p(Y)$ for all $\omega\in[0,2\pi]$
and that
$\omega\mapsto V_\omega$ is a continuous map from $[0,2\pi]$ to $L^p(Y)$.
For $v\equiv v(0)$, we show that $V^*_\omega\in L^\infty(Y)$ for all
$\omega\in(0,2\pi)$ and that
$\omega\mapsto V_\omega$ is a continuous map from $(0,2\pi)$ to $L^q(Y)$ for
all $q<\infty$.

First, suppose that $v(0)=0$, so $|v(y)|\le |v|_\eta|y|^\eta$.
It is well known that
$|f^\ell y|\ll (r(y)-\ell )^{-1/\gamma}$ for $\ell =1,\dots,r(y)$.
(See for example~\cite{LiveraniSaussolVaienti99}.)
Hence for $y\in Y$, $\omega\in[0,2\pi]$,
\begin{align*}
|V^*_\omega(y)| & \le 
	\sum_{\ell=0}^{r(y)-1}|v(f^\ell y)| \le
	\sum_{\ell=0}^{r(y)-1}|v|_{\eta}|f^\ell y|^\eta\ll
	\sum_{\ell=1}^{r(y)-1}(r(y)-\ell)^{-\eta/\gamma}\ll 
r(y)^{1-\eta/\gamma}.
\end{align*}
If $p<1/(\gamma-\eta)$, then $p(1-\eta/\gamma)<1/\gamma$ and
\begin{align*} 
\int_Y|V^*_\omega(y)|^p\,d\mu_Y
\ll \int_Y r^{p(1-\eta/\gamma)}\,d\mu_Y
<\infty.
\end{align*}

Next, we recall the estimate $|e^{ix}-1|\le 2\min\{1,|x|\}\le 2|x|^\eps$, which holds for all $x\in\R$, $\eps\in[0,1]$.
For all $y\in Y$, $\omega_0,\omega\in[0,2\pi]$,
\begin{align*}
	|V_\omega(y)-V_{\omega_0}(y)| & \le 
	\max_{0\le \ell<r(y)} |e^{i\ell\omega}-e^{i\ell\omega_0}|
	\sum_{\ell=0}^{r(y)-1} |v|_\eta|v(f^\ell y)|^\eta
	\\ & \ll r(y)^\eps|\omega-\omega_0|^\eps \sum_{\ell=0}^{r(y)-1}(r(y)-\ell)^{-\eta/\gamma}
	\ll r(y)^{1-\eta/\gamma+\eps}|\omega-\omega_0|^\eps,
\end{align*}
so for $\eps$ sufficiently small, $|V_\omega-V_{\omega_0}|_p\ll |\omega-\omega_0|^\eps$.

It remains to consider the case $v\equiv v(0)$.  We have
\[
	V^*_\omega(y)=|v(0)|\max_{0\le \ell<r(y)}|(1-e^{i\omega \ell})/(1-e^{i\omega})|\le 
2|v(0)||1-e^{i\omega}|^{-1},
\]
for all $y\in Y$, $\omega\in(0,2\pi)$.
Moreover, for $\omega_0,\omega\in(0,2\pi)$, regarding $\omega_0$ as fixed,
$|V_\omega(y)-V_{\omega_0}(y)|\le g_1(\omega,y)+g_2(\omega)$
where
\begin{align*}
& g_1(\omega,y)  = |v(0)||1-e^{i\omega_0}|^{-1} |e^{i\omega r(y)}-e^{i\omega_0 r(y)}|, \\
& g_2(\omega)  =2|v(0)||(1-e^{i\omega})^{-1}-(1-e^{i\omega_0})^{-1}|.
\end{align*}
Clearly, $g_2(\omega)\to0$ as $\omega\to\omega_0$.
Also, taking $\eps=1/q$, we have that $g_1(\omega,y)\ll r(y)^{1/q}|\omega-\omega_0|^{1/q}$, so
$\int_Y|g_1(\omega,y)|^q\,d\mu_Y\ll |r|_1|\omega-\omega_0|$.
It follows that $g_1(\omega,\cdot)\to0$ in $L^q(Y)$, and hence similarly for
$|V_\omega-V_{\omega_0}|$, as $\omega\to\omega_0$.
This completes the proof.
\end{proof}

\begin{pfof}{Theorem~\ref{thm-LSV}}
	As already mentioned, the case $\gamma\in(0,\frac12)$ follows directly from Theorem~\ref{thm-main}(a).

	When $\gamma\in[\frac12,1)$, we require that $\eta>(3\gamma-1)/2$.  Without loss, $\eta\in((3\gamma-1)/2,\gamma)$.  
		Note that $1/(\gamma-\eta)>2/(1-\gamma)>2$.
		By Proposition~\ref{prop-LSV},
$\omega\mapsto V_\omega$ is a continuous map from $(0,2\pi)$ to $L^2(Y)$.

Let $a=1/(\gamma+\eps)$, $b=1/(1-\gamma-\eps)$ where $\eps\in(0,1-\gamma)$.
		Then $r\in L^a(Y)$.
		Since $1/(\gamma-\eta)>2/(1-\gamma)$, it follows from Proposition~\ref{prop-LSV} that $V^*_\omega\in L^{2/(1-\gamma-\delta)}(Y)$ for $\delta>0$ sufficiently small.  Hence we can choose $p>2$, $\eps>0$ such that
		$V^*_\omega\in L^{bp}(Y)$ for all $\omega\in(0,2\pi)$.
Now apply Theorem~\ref{thm-main}(b).
\end{pfof}

\section{Proof of Theorem~\ref{thm-main}}
\label{sec-proof}

In Subsection~\ref{sec-induced}, we prove a version of our main results
for the induced system $F=f^r:Y\to Y$.  This result is lifted to the original system $f:Y\to Y$
in Subsection~\ref{sec-SYS}.
In Subsection~\ref{sec-subproof}, we complete the proof of Theorem~\ref{thm-main}.

\subsection{The induced system}
\label{sec-induced}

Let $f:X\to\ X$ be a nonuniformly expanding map as in Section~\ref{sec-NUE} with induced map $F=f^r:Y\to Y$ and partition $\alpha$.
Define $r_n=\sum_{j=0}^{n-1}r\circ F^j$.   The {\em induced power spectrum}
$S^Y:(0,2\pi)\to\R$ is given by
\begin{align*}
	S^Y(\omega) & =\lim_{n\to\infty}n^{-1}\int_Y\Bigl|\sum_{j=0}^{n-1}e^{i\omega r_j}V_{\omega}\circ F^j\Bigl|^2\,d\mu_Y.
\end{align*}

We can now state the main result in this subsection.

\begin{lemma}   \label{lem-induced}
Suppose that $r\in L^1(Y)$ and
that $\omega\mapsto V_\omega$ is continuous as a function from $(0,2\pi)$ to $L^2(Y)$.
\begin{itemize}
\item[(a)]
The pointwise limit
$S^Y:(0,2\pi)\to[0,\infty)$ exists and is continuous.
\item[(b)] Typically, $\{\omega\in(0,2\pi):S^Y(\omega)=0\}=\emptyset$.
\end{itemize}
\end{lemma}

In the remainder of this subsection, we prove Lemma~\ref{lem-induced}.

If $a_0,\dots,a_{n-1}\in\alpha$, we define the $n$-cylinder 
$[a_0,\dots,a_{n-1}]=\bigcap_{j=0}^{n-1}F^{-j}a_j$.
Fix $\theta\in(0,1)$ and define the symbolic 
metric $d_\theta(x,y)=\theta^{s(x,y)}$ where the {\em separation time}
$s(x,y)$ is the least integer $n\ge0$ such that $x$ and $y$ lie in distinct $n$-cylinders.
For convenience we rescale the metric $d$ on $X$ so that $\diam(Y)\le1$.

\begin{prop} \label{prop-d}
Let $\eta\in(0,1]$ and fix $\theta\in[\lambda^{-\eta},1)$.  Then 
$d(x,y)^\eta\le d_\theta(x,y)$ for all $x,y\in Y$.
\end{prop}

\begin{proof}
Let $n=s(x,y)$.  By condition~(2),
\[
1\ge \diam Y\ge d(F^nx,F^ny)\ge \lambda^nd(x,y)\ge (\theta^{1/\eta})^{-n}d(x,y).
\]
Hence $d(x,y)^\eta\le  \theta^n=d_\theta(x,y)$.
\end{proof}

An observable $\phi:Y\to\R$ is {\em Lipschitz} if $\|\phi\|_\theta=|\phi|_\infty+|\phi|_\theta<\infty$ where
$|\phi|_\theta=\sup_{x\neq y}|\phi(x)-\phi(y)|/d_\theta(x,y)$.
The set $F_\theta(Y)$ of Lipschitz observables is a Banach space.
More generally, we say that $\phi:Y\to\R$ is {\em locally Lipschitz}, and write
$\phi\in F_\theta^{\rm loc}(Y)$, if $\phi|_a\in F_\theta(a)$ for each $a\in\alpha$.  Accordingly, we define
$D_\theta\phi(a)=\sup_{x,y\in a:\,x\neq y}|\phi(x)-\phi(y)|/d_\theta(x,y)$.

\begin{prop}  \label{prop-DV}
Let $v:X\to\R$ be a $C^\eta$ function, $\eta\in(0,1]$.
Set $\theta=\lambda^{-\eta}$.   Then 
$V_\omega\in F_\theta^{\rm loc}(Y)$ for all $\omega\in[0,2\pi]$, and
there is a constant $C\ge1$ such that
\begin{align*}
 D_\theta V_\omega(a)\le C|v|_\eta r(a),
\end{align*}
for all $\omega\in[0,2\pi]$, $a\in\alpha$.
\end{prop}

\begin{proof}
Let $y,y'\in a$.  Then $r(y)=r(y')=r(a)$.  By condition~(3) and Proposition~\ref{prop-d},
\begin{align*}
& |V_\omega(y)-V_\omega(y')|
 \le \sum_{\ell=0}^{r(a)-1}|v(f^\ell y)-v(f^\ell y')|
\le |v|_\eta\sum_{\ell=0}^{r(a)-1}d(f^\ell y,f^\ell y')^\eta
\\ & \qquad  \le C^\eta |v|_\eta r(a)d(Fy,Fy')^\eta
\le C^\eta |v|_\eta r(a)d_\theta(Fy,Fy')
= C^\eta \theta^{-1}|v|_\eta r(a)d_\theta(y,y'),
\end{align*}
yielding the required estimate for $D_\theta V_\omega(a)$.
\end{proof}

The transfer operator $P:L^1(Y)\to L^1(Y)$ corresponding to $F$ is given by
$\int_Y P\phi\,\psi\,d\mu_Y=\int_Y \phi\,\psi\circ F\,d\mu_Y$ for all $\psi\in L^\infty$.
It can be shown that $(P\phi)(y)=\sum_{a\in\alpha}g(y_a)\phi(y_a)$
where $y_a$ denotes the unique preimage of $y$ in $a$ under $F$ and $\log g$ is the potential.
Moreover, there exists a constant $C_1$ such that
\begin{align} \label{eq-GM}
g(y)\le C_1\mu_Y(a), \quad\text{and}\quad |g(y)-g(y')|\le C_1\mu_Y(a)d_\theta(y,y'),
\end{align}
for all $y,y'\in a$, $a\in\alpha$.

For $\omega\in[0,2\pi]$, we define the twisted transfer operator
$P_\omega:L^1(Y)\to L^1(Y)$ given by
$P_\omega v=P(e^{-i\omega r}v)$.

\begin{prop} \label{prop-P}   
Let $J\subset(0,2\pi)$ be a closed subset.
Viewing $P_\omega$ as an operator on $F_\theta(Y)$, 
there exists $C\ge1$ and $\tau\in(0,1)$ such that
$\|P_\omega^n\|\le C\tau^n$ for all $\omega\in J$, $n\ge1$.
\end{prop}

\begin{proof}
This result is a combination of standard and elementary observations.
By~\cite[Theorem~2.4]{AaronsonDenker01},
$\omega\mapsto P_\omega$ is a continuous map from
$[0,2\pi]$ to $F_\theta(Y)$.  Hence it suffices to show that the spectral radius of $P_\omega$ is less than $1$ for $\omega\in(0,2\pi)$.

It is easily checked that the spectral radius of $P_\omega$ is at most $1$, and that the essential spectral radius is at most $\theta$ (see for example~\cite[Proposition~2.1]{AaronsonDenker01}).
Hence it remains to rule out eigenvalues on the unit circle.

Suppose for contradiction that $P_\omega v=e^{i\psi}v$ for some eigenfunction $v\in F_\theta(Y)$ and some 
$\psi\in[0,2\pi]$.
A calculation using the fact that $v\mapsto e^{i\omega r}v\circ F$ is the $L^2$ adjoint of $P_\omega$
(see for example~\cite[p.~429]{MN04}) shows that
$e^{i\omega r}v\circ F=e^{-i\psi}v$.  By ergodicity of $F$, $|v|$ is constant and hence
$v$ is nonvanishing.

Since $F|_a:a\to Y$ is onto for each $a$, there exists $y_a\in a$ with $Fy_a=y_a$.
Evaluating at $y_a$ and using the fact that $v(y_a)\neq0$, we obtain that
$e^{i\omega r(a)}=e^{-i\psi}$ for each $a\in\alpha$.
Hence $\omega (r(a)-r(b))=0\bmod 2\pi$ for all $a,b\in\alpha$.
 Since $\omega\in(0,2\pi)$, it is immediate that $\omega=2\pi p/q$ where $p,q$ are integers with $\gcd(p,q)=1$ and $1\le p<q$.
 But then $\frac{p}{q}(r(a)-r(b)=0\bmod\Z$ and so $q$ divides $r(a)-r(b)$ for all $a,b\in\alpha$.  This contradicts the assumption that
 $\gcd\{r(a)-r(b):a,b\in\alpha\}=1$.
\end{proof}

\begin{prop} \label{prop-PV}
Suppose that $r\in L^1(Y)$.  Then
\begin{itemize}
\item[(a)]  $P_\omega V_\omega\in F_\theta(Y)$ for all $\omega\in[0,2\pi]$
and $\sup_{\omega\in [0,2\pi]}\|P_\omega V_\omega\|_\theta<\infty$.
\item[(b)] $\sum_{n=1}^\infty \sup_{\omega\in J}
\bigl|\int_Y P_\omega^nV_\omega\,\bar V_\omega\,d\mu_Y\bigr|<\infty$
for any closed subset $J\subset(0,2\pi)$.
\end{itemize}
\end{prop}

\begin{proof}
(a) Write 
$(P_\omega \phi)(y)=\sum_{a\in\alpha}g(y_a)e^{-i\omega r(a)}\phi(y_a)$.
We use Propositions~\ref{prop-V} and~\ref{prop-DV} and the estimates~\eqref{eq-GM}.  First,
\[
|P_\omega V_\omega|_\infty\le \sum_{a\in\alpha} |1_ag|_\infty|1_aV_\omega|_\infty\le C_1 \sum_{a\in\alpha}\mu_Y(a)|v|_\infty r(a)=C_1|v|_\infty|r|_1.
\]
Also, $P_\omega V_\omega(y)-P_\omega V_\omega(y')=I_1+I_2+I_2$ where
\[
I_1 = \sum_{a\in\alpha} (g(y_a)-g(y_a'))e^{-i\omega r(a)}V_\omega(y_a), \quad
I_2 = \sum_{a\in\alpha} g(y_a')e^{-i\omega r(a)}(V_\omega(y_a)-V_\omega(y_a')).
\]
We have
\[
|I_1|\le C_1 \sum_{a\in\alpha} \mu_Y(a)d_\theta(y_a,y'_a)|v|_\infty r(a)=\theta C_1|v|_\infty|r|_1d_\theta(y,y'),
\]
and
\begin{align*}
|I_2| & \le C_1 \sum_{a\in\alpha} \mu_Y(a)D_\theta V_\omega(a) d_\theta(y_a,y_a')
\\ & \le \theta C_1C \sum_{a\in\alpha} \mu_Y(a)|v|_\eta r(a) d_\theta(y,y')
 =\theta C_1C |v|_\eta |r|_1 d_\theta(y,y').
\end{align*}
We deduce that $|P_\omega V_\omega|_\theta\ll (|v|_\infty+|v|_\eta)|r|_1$,
completing the proof of part~(a).

\vspace{1ex}
\noindent(b)
For $n\ge1$,
\begin{align*}
& \Bigl|\int_Y P_\omega^nV_\omega\,\bar V_\omega\,d\mu_Y\Bigr|
  \le 
|P_\omega^nV_\omega|_\infty \,|V_\omega|_1
\le \|P_\omega^{n-1}\| \|P_\omega V_\omega\|_\theta \,|V_\omega|_1.
\end{align*}
The result follows from Corollary~\ref{cor-V}, Proposition~\ref{prop-P} and part~(a),
 \end{proof}

\begin{pfof}{Lemma~\ref{lem-induced}(a)}
	Let $\omega\in (0,2\pi)$.
	For $0\le j<k$,
\begin{align*}
	\int_Y e^{-i\omega(r_k-r_j)}V_{\omega}\circ F^j\,\bar V_{\omega} & \circ F^k\,d\mu_Y 
 = \int_Y e^{-i\omega r_{k-j}\circ F^j}V_{\omega}\circ F^j\,\bar V_{\omega}\circ F^k\,d\mu_Y \\
 & 	=\int_Y e^{-i\omega r_{k-j}}V_{\omega}\,\bar V_{\omega}\circ F^{k-j}\,d\mu_Y 
	=\int_Y P_\omega^{k-j}V_{\omega}\,\bar V_{\omega}\,d\mu_Y. 
\end{align*}
Hence,
\begin{align*}
	&  \int_Y|\sum_{j=0}^{n-1}e^{i\omega r_j}V_{\omega}\circ F^j|^2\,d\mu_Y    =
	\sum_{j,k=0}^{n-1}\int_Y e^{i\omega(r_j-r_k)}V_{\omega}\circ F^j\,\bar V_{\omega}\circ F^k\,d\mu_Y \\
	& \quad\qquad = \sum_{j=0}^{n-1}\int_Y|V_{\omega}\circ F^j|^2\,d\mu_Y
	+ 2\sum_{0\le j<k<n}\Re\int_Y e^{-i\omega(r_k-r_j)}V_{\omega}\circ F^j\,\bar V_{\omega}\circ F^k\,d\mu_Y \\
	& \quad\qquad = n\int_Y|V_{\omega}|^2\,d\mu_Y
	+ 2\sum_{0\le j<k<n}\Re\int_Y P_\omega^{k-j}V_{\omega}\,\bar V_{\omega}\,d\mu_Y \\
	& \quad\qquad = n\int_Y |V_{\omega}|^2\,d\mu_Y
	+ 2\sum_{m=1}^{n-1}(n-m)\Re\int_Y P_\omega^mV_{\omega}\,\bar V_{\omega}\,\,d\mu_Y.
\end{align*}
Hence
\begin{align*}
	S^Y(\omega)=\int_Y|V_{\omega}|^2\,d\mu_Y
	+ 2\lim_{n\to\infty}\sum_{m=1}^{n-1}\bigl(1-\frac{m}{n}\bigr)\Re\int_Y P_\omega^mV_{\omega}\,\bar V_{\omega}\,d\mu_Y.
\end{align*}
By Proposition~\ref{prop-PV}(b), this converges uniformly on compact subsets of $(0,2\pi)$
to the sum 
\begin{align} \label{eq-unif} \nonumber
	S^Y(\omega) & =\int_Y|V_{\omega}|^2\,d\mu_Y+2\sum_{n=1}^\infty\Re\int_Y P_\omega^nV_{\omega}\,\bar V_{\omega}\,d\mu_Y \\
	& =\int_Y|V_{\omega}|^2\,d\mu_Y+2\sum_{n=1}^\infty\Re\int_Y e^{-i\omega r_n}V_{\omega}\,\bar V_{\omega}\circ F^n\,d\mu_Y.
\end{align}

Fix $n\ge1$ and let $I_\omega=e^{-i\omega r_n} V_{\omega}\,\bar V_{\omega}\circ F^n$.
Note that $|I_\omega|=|V_{\omega}|\,|V_{\omega}|\circ F^n$ and
$\int_Y |I_\omega|\,d\mu_Y\le |V_{\omega}|_2|V_{\omega}\circ F^n|_2
= |V_{\omega}|_2^2<\infty$ for each $\omega$.
We claim that $\omega\mapsto \int_Y I_\omega \,d\mu_Y$ is continuous on $(0,2\pi)$.
It then follows from uniform convergence of the series~\eqref{eq-unif} that 
$S^Y:(0,2\pi)\to[0,\infty)$ is continuous.

	To prove the claim,  fix $n\ge1$ and $\omega_*\in (0,2\pi)$.
	Let $\omega_k$ be a sequence in $(0,2\pi)$ converging to $\omega_*$.
We show that $\int_Y I_{\omega_k}\,d\mu_Y\to \int_Y I_{\omega_*}\,d\mu_Y$ as $k\to\infty$.
Certainly $I_{\omega_k}\to I_{\omega_*}$ pointwise.
Moreover,
\begin{align*}
\int_Y |I_{\omega_k}|\,d\mu_Y -\int_Y |I_{\omega_*}|\,d\mu_Y 
& = \int_Y(|V_{\omega_k}|-|V_{\omega_*}|)\,|V_{\omega_k}|\circ F^n\,d\mu_Y
\\ & \qquad  \qquad + \int_Y |V_{\omega_*}|\,(|V_{\omega_k}|-|V_{\omega_*}|)\circ F^n\,d\mu_Y,
\end{align*}
and so
\begin{align*}
\Bigl|\int_Y |I_{\omega_k}|\,d\mu_Y -\int_Y |I_{\omega_*}|\,d\mu_Y \Bigr|
& \le \bigl||V_{\omega_k}|-|V_{\omega_*}|\bigr|_2 |V_{\omega_k}|_2
+|V_{\omega_*}|_2\bigl||V_{\omega_k}|-|V_{\omega_*}|\bigr|_2 \\
& \le |V_{\omega_k}-V_{\omega_*}|_2(|V_{\omega_k}|_2+|V_{\omega_*}|_2) \\
& \le |V_{\omega_k}-V_{\omega_*}|_2(|V_{\omega_k}-V_{\omega_*}|_2+2|V_{\omega_*}|_2). 
\end{align*}
It follows that $\int_Y|I_{\omega_k}|\,d\mu_Y\to
\int_Y|I_{\omega_*}|\,d\mu_Y$ as $k\to\infty$.
By the dominated convergence theorem, 
$\int_Y I_{\omega_k}\,d\mu_Y\to \int_Y I_{\omega_*}\,d\mu_Y$ completing the proof of the claim.
\end{pfof}

\begin{pfof}{Lemma~\ref{lem-induced}(b)}
Let $\omega\in(0,2\pi)$ and define
 $\chi_\omega  =\sum_{j=1}^\infty P_\omega^jV_\omega$ and
 $\tilde V_\omega  = V_\omega+\chi_\omega- e^{i\omega r}\chi_\omega\circ F$.
 Now $\|P_\omega^jV_\omega\|_\theta\le \|P_\omega^{j-1}\| \|P_\omega V_\omega\|_\theta$, so it follows from 
 Propositions~\ref{prop-P} and~\ref{prop-PV}(a) that
$\chi_\omega$ is absolutely summable in $F_\theta(Y)$.  In particular
$\tilde V_\omega\in L^2(Y)$.

A calculation shows that $\tilde V_\omega\in\ker P_\omega$ and so
$|\sum_{j=1}^{n-1}e^{i\omega r_j}\tilde V_\omega\circ F^j|_2^2=n|\tilde V_\omega|_2^2$.   Moreover,  $\sum_{j=1}^{n-1}e^{i\omega r_j}\tilde V_\omega\circ F^j
= \sum_{j=1}^{n-1}e^{i\omega r_j}V_\omega\circ F^j +\chi_\omega-e^{i\omega r_n}\chi_\omega\circ F_\omega^n$ and so
\[
\Bigl|\,\bigl|\sum_{j=1}^{n-1}e^{i\omega r_j}V_\omega\circ F^j\bigr|_2-
\bigl|\sum_{j=1}^{n-1}e^{i\omega r_j}\tilde V_\omega\circ F^j\bigr|_2\,\Bigr|\le 2|\chi_\omega|_2.
\]
Hence $S^Y(\omega)=|\tilde V_\omega|_2^2$ for all $\omega\in (0,2\pi)$.
In particular, for a fixed $\omega\in (0,2\pi)$ we have that if $S^Y(\omega)=0$, then equivalently $\tilde V_\omega=0$ and so
\begin{align} \label{eq-zero}
V_\omega=e^{i\omega r}\chi_\omega\circ F-\chi_\omega.
\end{align}

It remains to exclude the possibility that~\eqref{eq-zero}
holds for some $\omega$.
Following~\cite[Section~3]{MG08}, 
let $y_0\in Y$ be a periodic point of period $p$ and let $y_n$ be a sequence with
$Fy_n=y_{n-1}$ and such that $d_\theta(y_{np},y_0)\le \theta^{np}$.
This ensures in particular that $r_p(y_{jp})=r_0(y_0)$ for all $j$.

Set $A_\omega=\sum_{j=0}^{p-1}e^{i\omega r_j}V_\omega\circ F^j$ 
and define 
$g(\omega)=\sum_{j=1}^\infty e^{-ij\omega r_p(y_0)}(A_\omega(y_{jp})-A_\omega(y_0))$.

Note that for each fixed $y_n$, the function $\omega\mapsto A_\omega(y_n)$
is a finite trigonometric polynomial and hence is analytic on $[0,2\pi]$.
We claim that $g:[0,2\pi]\to\C$ is analytic.
Suppose that $d_\theta(y,y')\ge p$.  Then 
\begin{align*}
|A_\omega(y)-A_\omega(y')|
& \le \sum_{j=0}^{p-1}\sum_{\ell=0}^{r(F^jy)-1}|v|_\theta d_\theta(F^jy,F^jy')
= |v|_\theta\sum_{j=0}^{p-1} r(F^jy)\theta^{-j}d_\theta(y,y')
\\ & \le  |v|_\theta \,\theta^{-p}(1-\theta)^{-1}r_p(y)d_\theta(y,y'),
\end{align*}
Hence
\begin{align*}
|A_\omega(y_{jp})-A_\omega(y_0)|
& \le  |v|_\theta \,\theta^{-p}(1-\theta)^{-1}r_p(y_0)\theta^j,
\end{align*}
proving the claim.

If~\eqref{eq-zero} holds, then
$A_\omega=e^{i\omega r_p}\chi_\omega\circ F^p-\chi_\omega$ and 
so 
\begin{align*}
& \sum_{j=1}^n e^{-ij\omega r_p(y_0)}A_\omega(y_{jp})=
\chi_\omega(y_0)-e^{-in\omega r_p(y_0)}\chi_\omega(y_{np}), \\
& \sum_{j=1}^n e^{-ij\omega r_p(y_0)}A_\omega(y_0)  =
\chi_\omega(y_0)-e^{-in\omega r_p(y_0)}\chi_\omega(y_0).
\end{align*}
Hence
\[
g(\omega)
= \lim_{n\to\infty} e^{-in\omega r_p(y_0)}(\chi_\omega(y_0)-\chi_\omega(y_{np}))=0.
\]

Still keeping $\omega$ fixed, we can perturb the value of $v$ at $y_1$ (say), independently of any other values of $v$ involved in the computation of $g$, so that $g(\omega)\neq0$.  Hence typically~\eqref{eq-zero} does not hold (and so $S^Y(\omega)$ is nonzero) for any fixed value of $\omega$.

Considering two such analytic functions $g_1$ and $g_2$ (for two distinct periodic points)
we can perturb so that $g_1$ and $g_2$ have no common zeros on $[0,2\pi]$ and hence $S^Y$
is nonvanishing on $(0,2\pi)$.   By considering infinitely many periodic points, and hence infinitely many functions of the form $g$, we obtain infinitely many independent obstructions to the existence of an $\omega\in(0,2\pi)$ such that $S^Y(\omega)=0$.
\end{pfof}

\subsection{Relation between $S^Y$ and $S$}
\label{sec-SYS}

In this subsection, we relate the power spectrum $S^Y$ of the induced map $F:Y\to Y$ with the power spectrum $S$ of the underlying nonuniformly expanding map $f:X\to X$.
Let $\bar r= \int_Y r\,d\mu_Y$.
We say that $\omega$ is an irrational angle if 
$\omega\in[0,2\pi]\setminus\pi\Q$.  

\begin{lemma} \label{lem-SYS}
Let $\omega$ be an irrational angle.  
Suppose either that $r\in L^{2+}(Y)$, or that $r\in L^a(Y)$ and 
that $V^*_\omega\in L^{bp}(Y)$ for all $\omega\in (0,2\pi)$, where
$a\in(1,\infty]$, $1/a+1/b=1$ and $p>2$.
Then $S(\omega)=S^Y(\omega)/\bar r$.
\end{lemma}

From now on, we work with a fixed irrational angle $\omega\in(0,2\pi)\setminus\pi\Q$.
We suppose throughout that $v:X\to\R$ is H\"older.

Let $d\varphi$ denote Haar measure on $S^1$,
Consider the circle extensions
\begin{align*}
& f_\omega:X\times S^1\to X\times S^1, \qquad
f_\omega(x,\varphi)=(fx,\varphi+\omega), \\
& F_\omega:Y\times S^1\to Y\times S^1, \qquad
F_\omega(y,\varphi)=(Fy,\varphi+\omega r(y)),
\end{align*}
with invariant probability measures $\nu=\mu\times d\varphi$ and
$\nu_Y=\mu_Y\times d\varphi$ respectively.

Recall that $F=f^r:Y\to Y$ is the induced map obtained from $f:X\to X$ with return time $r:Y\to\Z^+$.
Extend $r$ to a return time on $Y\times S^1$ by setting $r(y,\varphi)=r(y)$. Then $F_\omega=f_\omega^r$ is the induced map obtained from $f_\omega$.

Let $v:X\to \R$ be an observable.
We associate to $v$ the observable $u:X\times S^1\to\C$ given by
$u(x,\varphi)=e^{i\varphi}v(x)$.  
This leads to the induced observable
$U_\omega:Y\times S^1\to\C$ given by $U_\omega(q)=\sum_{\ell=0}^{r(q)-1}u(f_\omega^\ell q)$.
Note that
\[
U_\omega(y,\varphi)=\sum_{\ell=0}^{r(y)-1}u(f^\ell y,\varphi+\ell\omega)
=e^{i\varphi}\sum_{\ell=0}^{r(y)-1}e^{i\ell\omega}v(f^\ell y)
=e^{i\varphi}V_\omega(y),
\]
and similarly that
\[
\sum_{j=0}^{n-1}U_\omega\circ F_\omega^j(y,\varphi)=e^{i\varphi}\sum_{j=0}^{n-1}e^{i\omega r_j(y)}V_\omega\circ F^j(y).
\]

\begin{prop}   \label{prop-wip}
	Let $\omega$ be an irrational angle.
Suppose that $r\in L^{1+}(Y)$ and that $V^*_\omega\in L^2(Y)$.  Then
\[
n^{-1}\Bigl|\sum_{j=0}^{n-1}u\circ f_\omega^j|^2\to_d Z_\omega \quad\text{on $(X\times S^1,\nu)$},
\]
where $Z_\omega$ is a random variable with $\E Z_\omega=S^Y(\omega)/\bar r$.
\end{prop}

\begin{proof}
Since $\omega\in [0,2\pi]\setminus\pi\Q$, the circle extension $f_\omega:X\times S^1\to X\times S^1$ is ergodic, and we are in a position to apply~\cite[Theorem~1.10]{GMsub}.  Note that $G,\phi,V,V^*,f_h$ in~\cite{GMsub} correspond to
$S^1,u,V_\omega,V^*_\omega,f_\omega$ here.  

By~\cite[Theorem~1.10]{GMsub}, we obtain a functional central limit theorem (weak invariance principle) as follows.
Define 
$W_{n,\omega}(t)
=n^{-1/2}\sum_{j=0}^{[nt]-1}u\circ f_\omega^j$ for $t=0,1/n,\dots,1$ and linearly interpolate to obtain $W_{n,\omega}\in C([0,1],\R^2)$.
Then  $W_{n,\omega}\to_w W_\omega$ in $C([0,1],\R^2)$ 
on $(X\times S^1,\nu),$ where
$W_\omega$ is a two-dimensional Brownian motion with some covariance matrix $\Sigma_\omega$.
Moreover, it follows from the proof of~\cite[Theorem~1.10]{GMsub}
(see the statements of~\cite[Theorems~2.1 and~3.3]{GMsub})
that $\Sigma_\omega=\hat\Sigma_\omega/\bar r$ where
\[
\hat\Sigma_\omega =
\lim_{n\to\infty}n^{-1}\int_{Y\times S^1}
\Bigl(\sum_{j=0}^{n-1}U_\omega\circ F_\omega^j\Bigr)
\otimes \Bigl(\sum_{j=0}^{n-1}U_\omega\circ F_\omega^j\Bigr)
\,d\nu_Y.
\footnote{For $a,b\in\C\cong\R^2$, we define
$a\otimes b=ab^T=\left(\begin{array}{cc} \Re a\Re b & \Re a\Im b \\ \Im a\Re b & \Im a \Im b \end{array}\right)$.}
\]

Consider the functional $\chi:C([0,1],\R^2)\to\R$, $\chi(g)=|g(1)|^2$.
By the continuous mapping theorem,
$\chi(W_{n,\omega})\to_d \chi(W_\omega)$, so 
\[
n^{-1}|\sum_{j=0}^{n-1}u\circ f_\omega^j|^2\to_d Z_\omega,
\quad\text{where $Z_\omega=|W_\omega(1)|^2$}.
\]
In particular, 
\begin{align*}
\E Z_\omega=\E|W_\omega(1)|^2=\Sigma_\omega^{11}+\Sigma_\omega^{22}
=(\hat\Sigma_\omega^{11}+\hat\Sigma_\omega^{22})/\bar r.
\end{align*}

Finally,
\begin{align*}
S^Y(\omega) & =\lim_{n\to\infty}n^{-1}\int_Y|\sum_{j=0}^{n-1}e^{i\omega r_j}V_\omega\circ F^j|^2\,d\mu_Y \\ & =
\lim_{n\to\infty}n^{-1}\int_{Y\times S^1}|\sum_{j=0}^{n-1}U_\omega\circ F_\omega^j|^2\,d\nu_Y
=\hat\Sigma_\omega^{11}+\hat\Sigma_\omega^{22},
\end{align*}
completing the proof.
\end{proof}

\begin{prop}  \label{prop-moment}
	Suppose that either
\begin{itemize}
	\item[(a)] $r\in L^{2+}(Y)$ and choose $p>2$ such that $r\in L^{\frac{p}{2}+1+}(Y)$, or
	\item[(b)] $r\in L^a(Y)$ and 
$V^*_\omega\in L^{bp}(Y)$ for all $\omega\in (0,2\pi)$, where
$a\in(1,\infty]$, $1/a+1/b=1$, and $p>2$.
\end{itemize}
Then there is a constant $C\ge1$ such that 
$|\sum_{j=0}^{n-1}u\circ f_\omega^j|_p\le Cn^{1/2}$.
\end{prop}

\begin{proof}
	The method in both cases is to obtain a martingale coboundary decomposition~\cite{Gordin69} and then to apply Burkholder's inequality~\cite{Burkholder73} to the martingale part.   In case~(a), the decomposition is done on $X\times S^1$ following~\cite{MN08}.  In case~(b), we pass to a tower extension and reduce to the induced system on $Y\times S^1$.

	Case (a): 
	By Markov's inequality, the assumption $r\in L^{\frac{p}{2}+1+}(Y)$ guarantees that $\mu(r>n)=O(n^{-(\beta+1)})$ for some $\beta>p/2$.
Let $L:L^1(X)\to L^1(X)$ and $L_\omega:L^1(X\times S^1)\to L^1(X\times S^1)$
be the transfer operators corresponding to $f$ and  $f_\omega$ respectively.
Regard $u$ as fixed, and
let $u'\in L^\infty(X\times S^1)$ be of the form $u'(y,\varphi)=e^{i\varphi}v'(x)$ where $v'\in L^\infty(X)$.
By~\cite{BHM05},
$\int_{X\times S^1} L_\omega^n u\,\bar u'\,d\nu=
\int_{X\times S^1} u\,\bar u'\circ f_\omega^n\,d\nu=O(n^{-\beta}|u'|_\infty)$.
Equivalently $\int_X L^nv \,\bar v'\,d\mu=O(n^{-\beta}|v'|_\infty)$.
By duality
$\int_{X\times S^1}|L_\omega^nu| \,d\nu=
\int_X|L^nv| \,d\mu=O(n^{-\beta})$.

Now we proceed as in the proof of~\cite[Theorem~3.1]{MN08}.
Since
$|L_\omega^nu|_1 =O(n^{-\beta})$ and $|L_\omega^nu|_\infty =O(1)$,
it follows by interpolation that
$|L_\omega^nu|_q$ is summable for $q<\beta$, in particular for $q=p/2$.
Hence 
$\chi_\omega=\sum_{n=1}^\infty L_\omega^nu\in L^{p/2}(X\times S^1)$
and we obtain $u=\tilde u_\omega+\chi_\omega\circ f_\omega-\chi_\omega$
where $\tilde u_\omega\in\ker L_\omega$.
Continuing as in~\cite{MN08} (in particular, see~\cite[equation~(3.1)]{MN08}) we obtain the desired result.

\vspace{1ex} \noindent  Case (b):
Define the Young tower~\cite{Young99}
\begin{align*} 
	\Delta=\{(y,\varphi,\ell)\in Y\times S^1\times\Z:0\le\ell<r(y)\},
\end{align*}
and the tower map $\hat f_\omega:\Delta\to\Delta$,
\[
	\hat f_\omega(y,\varphi,\ell)=\begin{cases} (y,\varphi+\omega,\ell+1), & \ell\le r(y)-2
\\ (Fy,\varphi+\omega,0), & \ell=r(y)-1 \end{cases}
\]
with invariant probability measure
$\nu_\Delta=(\nu\times{\rm counting})/\bar r$.  

The projection $\pi:\Delta\to X\times S^1$ given by $\pi(y,\varphi,\ell)=(f^\ell y,\varphi)$ is a measure-preserving semiconjugacy between $\hat f_\omega$ and $f_\omega$, with
$\nu=\pi_*\nu_\Delta$.

Let $\hat u=u\circ \pi$.
Then $\int_{X\times S^1} |\sum_{j=0}^{n-1}u\circ f_\omega^j|^p\,d\nu=\int_{\Delta} |\sum_{j=0}^{n-1}\hat u\circ\hat f_\omega^j|^p\,d\nu_\Delta$.

Next, let $N_n:\Delta\to\{0,1,\dots,n\}$ be the number of laps by time $n$,
\[
	N_n(y,\varphi,\ell)=\#\{j\in\{1,\dots,n\}:\hat f_\omega^j(y,\varphi,\ell)\in Y\times S^1\times\{0\}\}.
\]
Then
\[
	e^{i\ell\omega}\sum_{j=0}^{n-1}\hat u\circ \hat f_\omega^j(y,\varphi,\ell)=
	\sum_{k=0}^{N_n(y,\varphi,\ell)-1}U_\omega\circ F_\omega^k(y,\varphi)+H_\omega\circ \hat f_\omega^n(y,\varphi,\ell)-H_\omega(y,\varphi,\ell)
\]
where $H_\omega(y,\varphi,\ell)=e^{i\varphi}\sum_{\ell'=0}^{\ell-1} e^{i\ell'\omega}v(f^{\ell'}y)$.
Note that $|H_\omega(y,\varphi,\ell)|\le V^*_\omega(y)$.

Now
\begin{align} \label{eq-H} \nonumber
	\int_\Delta |H_\omega\circ \hat f_\omega^n|^p\,d\mu_\Delta & =
	\int_\Delta |H_\omega|^p\,d\mu_\Delta=(1/\bar r)\int_Y\sum_{\ell=0}^{r(y)-1}|H_\omega(y,\varphi,\ell)|^p\,d\nu_Y(y,\varphi)
\\ & \le \int_Y r \ |V^*_\omega|^p\,d\mu_Y
\le |r|_a |{V^*_\omega}^p|_b = |r|_a |V^*_\omega|_{bp}^p.
\end{align}

Next, by H\"older's inequality,
 \begin{align*}
	 \int_{\Delta} & 
\Bigl|\sum_{k=0}^{N_n-1}U_\omega\circ F_\omega^k\Bigr|^p\,d\nu_\Delta
 \le \int_{\Delta}
\max_{j\le n}\Bigl|\sum_{k=0}^{j-1}U_\omega\circ F_\omega^k\Bigr|^p\,d\nu_\Delta
\\
& =(1/\bar r) \int_{Y\times S^1}  r \
\max_{j\le n}\Bigl|\sum_{k=0}^{j-1}U_\omega\circ F_\omega^k\Bigr|^p\,d\nu_Y
 \le   
|r|_{L^a(Y)} \  \Biggl|\max_{j\le n}\Bigl|\sum_{k=0}^{j-1}U_\omega\circ F_\omega^k\Bigr|\Biggr|_{L^{bp}(Y\times S^1)}^p.
     \end{align*}

Let $Q_\omega:L^1(Y\times S^1)
\to L^1(Y\times S^1)$ denote the transfer operator corresponding
to $F_\omega:Y\times S^1\to Y\times S^1$.
For an observable $U:Y\times S^1\to\C$ of the form
$U(y,\varphi)=e^{i\varphi}V(y)$, we have
$(Q_\omega U)(y,\varphi)=e^{i\varphi}(P_\omega V)(y)$.

By Propositions~\ref{prop-P} and~\ref{prop-PV}(a),
\[
\chi_\omega=\sum_{n=1}^\infty Q_\omega^nU_\omega
=e^{i\varphi}\sum_{n=1}^\infty P_\omega^nV_\omega\in L^\infty(Y\times S^1).
\]
Hence we can write
     \[
	     U_\omega= m_\omega+\chi_\omega\circ F_\omega-\chi_\omega,
     \]
     where $m_\omega\in\ker Q_\omega$.  
     Note also by the hypothesis on $V^*_\omega$ that
     $U_\omega\in L^{bp}(Y\times S^1)$ and hence
     $m_\omega\in L^{bp}(Y\times S^1)$ where $bp>2$.
     By Burkholder's inequality~\cite{Burkholder73},
     \[
	     \Bigl|\max_{j\le n}\bigl|\sum_{k=0}^{j-1}m_\omega\circ F_\omega^k\bigr|\Bigr|_{L^{bp}(Y\times S^1)}\ll |m_\omega|_{L^{bp}(Y\times S^1)} \ n^{1/2},
     \]
     and so
     $\bigl|\max_{j\le n}|\sum_{k=0}^{j-1}U_\omega\circ F_\omega^k|\bigr|_{L^{bp}(Y\times S^1)}\ll n^{1/2}$.  Hence we have shown that
\begin{align} \label{eq-maxU}
	\Bigl|\sum_{k=0}^{N_n-1}U_\omega\circ F_\omega^k\Bigr|_{L^p(\Delta)}\le |r|_a^{1/p} \
\Biggl|\max_{j\le n}\Bigl|\sum_{k=0}^{j-1}U_\omega\circ F_\omega^k\Bigr|\Biggr|_{L^{bp}(Y\times S^1)}\ll n^{1/2}.
\end{align}

     By the triangle inequality, it follows from~\eqref{eq-H} and~\eqref{eq-maxU} that
\[
	\Bigl|\sum_{j=0}^{n-1}u\circ f_\omega^j\Bigr|_{L^p(X\times S^1)}=\Bigl|\sum_{j=0}^{n-1}\hat u\circ\hat f_\omega^j\Bigr|_{L^p(\Delta)}
	\ll n^{1/2},
\]
as required.
\end{proof}

\begin{pfof}{Lemma~\ref{lem-SYS}}
	The bound on moments in Proposition~\ref{prop-moment} together with
the distributional limit law in Proposition~\ref{prop-wip} implies convergence of lower moments,
(see for example~\cite{MTorok12}), 
and so 
\[
\lim_{n\to\infty}n^{-1}\Bigl|\sum_{j=0}^{n-1} u\circ f_\omega^j\Bigr|^2_{L^2(X\times S^1)}=
S^Y(\omega)/\bar r.
\]

Now $\sum_{j=0}^{n-1} u(f_\omega^j(x,\varphi))=e^{i\varphi}
\sum_{j=0}^{n-1} e^{ij\omega}v(f^j(x))$ and so
$|\sum_{j=0}^{n-1} u\circ f_\omega^j|_{L^2(X\times S^1)}=
|\sum_{j=0}^{n-1} e^{ij\omega}v\circ f^j|_{L^2(X)}$.
Hence
\[
S(\omega)=\lim_{n\to\infty}n^{-1}\Bigl|\sum_{j=0}^{n-1} e^{ij\omega}v\circ f^j\Bigr|^2_{L^2(X)}=	
S^Y(\omega)/\bar r,
\]
as required.
\end{pfof}

\subsection{Completion of the proof}
\label{sec-subproof}

First assume the hypotheses of Theorem~\ref{thm-main}(b).
By Lemma~\ref{lem-induced}, $S^Y$ exists and is continuous on $(0,2\pi)$,
and typically $S^Y$ is nonvanishing on $(0,2\pi)$.
By Lemma~\ref{lem-SYS}, $S$ coincides with $S^Y/\bar r$ almost everywhere on 
$(0,2\pi)$.
Theorem~\ref{thm-main}(b) follows immediately.

Under the hypotheses of Theorem~\ref{thm-main}(a) we have the same properties, but in addition
	$S(\omega)$ is continuous on $(0,2\pi)$ by the Wiener-Khintchine Theorem, so it follows that 
	$S(\omega)=S^Y(\omega)/\bar r$ for all $\omega\in(0,2\pi)$.
Hence $S$ is typically nonvanishing on $(0,2\pi)$.
Moreover, by Wiener-Khintchine, $S$ extends to a continuous function
$S_0(\omega)=\sum_{k=-\infty}^\infty e^{ik\omega}\rho(k)$ on $[0,2\pi]$.
By the Green-Kubo formula, $S_0(0)=S_0(2\pi)$ coincides with the variance
$\sigma^2=\lim_{n\to\infty}n^{-1}\int_X|\sum_{j=0}^{n-1}v_0\circ f^j|^2\,d\mu$
where $v_0=v-\int_X v\,d\mu$.
Typically $\sigma^2>0$, see for example~\cite[Remark~2.11]{MN05}.
Hence typically $S_0$ is bounded away from zero on $[0,2\pi]$ and so
$S$ is bounded away from zero on $(0,2\pi)$.
This completes the proof of Theorem~\ref{thm-main}(a).

\paragraph{Acknowledgements}
This research was supported in part by an International Research Collaboration Award at the University of Sydney. GAG acknowledges funding from the Australian Research Council. The research of IM was supported in part by a European Advanced Grant {\em StochExtHomog} (ERC AdG 320977).

\end{document}